\documentclass[letterpaper, 12pt]{amsart}

\usepackage[margin=1in]{geometry}
\usepackage{amsmath,amssymb,amsthm}
\usepackage{fouriernc}
\usepackage{mathtools}
\usepackage[colorlinks=true, linkcolor=blue, citecolor=blue,
pagebackref=true]{hyperref}

\newtheorem{theorem}{Theorem}
\newtheorem{lemma}{Lemma}
\newtheorem{proposition}{Proposition}

\theoremstyle{remark}
\newtheorem*{remark*}{Remark}

\def\eps{{\varepsilon}}

\newcommand{\p}{\mathbf p}
\newcommand{\q}{\mathbf q}

\newcommand{\br}{\mathbf r}
\newcommand{\bv}{\mathbf v}
\newcommand{\x}{\mathbf x}
\newcommand{\NN}{\mathbb N}
\newcommand{\PP}{\mathbb P}
\newcommand{\RR}{\mathbb R}
\newcommand{\TT}{\mathbb T}
\newcommand{\ZZ}{\mathbb Z}
\newcommand{\bone}{\mathbf 1}

\newcommand{\I}{\textrm{I}}

\DeclarePairedDelimiter{\abs}{\lvert}{\rvert}
\DeclarePairedDelimiter{\parens}{\lparen}{\rparen}

\DeclarePairedDelimiter{\set}{\lbrace}{\rbrace}

\DeclareMathOperator{\Leb}{Leb}

\title{The Duffin--Schaeffer conjecture for systems of linear forms}

\author{Felipe A.~Ram{\'i}rez}

\address{Wesleyan University \\ Middletown, CT, USA}
\email{framirez@wesleyan.edu}

\date{}

\makeatletter
\@namedef{subjclassname@2020}{\textup{2020} Mathematics Subject Classification}
\makeatother

\subjclass[2020]{Primary: 11J83, 11J13, 11K60} \keywords{Metric number
  theory, Diophantine approximation, Duffin--Schaeffer conjecture,
  linear forms}

\begin{document}

\frenchspacing

\begin{abstract}
  We extend the Duffin--Schaeffer conjecture to the setting of systems
  of $m$ linear forms in $n$ variables. That is, we establish a
  criterion to determine whether, for a given rate of approximation,
  almost all or almost no $n$-by-$m$ systems of linear forms are
  approximable at that rate using integer vectors satisfying a natural
  coprimality condition. When $m=n=1$, this is the classical
  1941 Duffin--Schaeffer conjecture, which was proved in 2020 by
  Koukoulopoulos and Maynard. Pollington and Vaughan proved the
  higher-dimensional version, where $m>1$ and $n=1$, in 1990. The
  general statement we prove here was conjectured in 2009 by
  Beresnevich, Bernik, Dodson, and Velani. For approximations with no
  coprimality requirement, they also conjectured a generalized version
  of Catlin's conjecture, and in 2010 Beresnevich and Velani proved
  the $m>1$ cases of that. Catlin's classical conjecture, where
  $m=n=1$, follows from the classical Duffin--Schaeffer
  conjecture. The remaining cases of the generalized version, where
  $m=1$ and $n>1$, follow from our main result. Finally, through the
  Mass Transference Principle, our main results imply their Hausdorff
  measure analogues, which were also conjectured by Beresnevich
  \emph{et al} (2009).
\end{abstract}

\maketitle


\section{Introduction}
\label{sec:intro}

\subsection{Results}
\label{sec:main-results}

We establish the following result, which was conjectured by
Beresnevich, Bernik, Dodson, and Velani in~\cite[2009]{BBDV}. It is the
generalization of the Duffin--Schaeffer conjecture to the setting of
systems of linear forms.

\begin{theorem}\label{thm:multiv}
  Let $m$ and $n$ be positive integers. If $\psi:\ZZ^n\to\RR_{\geq 0}$
  is a function such that
  \begin{equation}\label{eq:maindiv}
    \sum_{\q\in\ZZ^n\setminus\set{0}}\parens*{\frac{\varphi(\gcd(\q))\psi(\q)}{\gcd(\q)}}^m = \infty,
  \end{equation}
  then for Lebesgue-almost every $\x\in \operatorname{Mat}_{n\times m}(\RR)$
  there are infinitely many $(\p,\q)\in\ZZ^m\times\ZZ^n$ such that
  \begin{equation}\label{eq:1}
    \abs*{\q\x - \p} < \psi(\q)
  \end{equation}
  and
  \begin{equation}
    \label{eq:coprimality}
    \gcd(p_i, \q)=1 \quad\textrm{for every}\quad i=1, \dots, m,
  \end{equation}
  both hold. Conversely, if the sum converges, then there are almost
  no such $\x$.
\end{theorem}

\begin{remark*}
  Throughout this paper, $\abs{\cdot}$ denotes sup-norm, $\varphi$ is
  Euler's totient function, and $\gcd(\q)$ is the greatest common
  divisor of the components of $\q$, which we implicitly regard as a
  row vector.
\end{remark*}

When $m=n=1$, Theorem~\ref{thm:multiv} is the classical
Duffin--Schaeffer conjecture~\cite[1941]{duffinschaeffer}, which was
proved by Koukoulopoulos and Maynard~\cite[2020]{KMDS}. When $n=1$ and
$m>1$, Theorem~\ref{thm:multiv} gives the higher-dimensional
``simultaneous'' version of the Duffin--Schaeffer conjecture, which
was proved by Pollington and
Vaughan~\cite[1990]{PollingtonVaughan}. The cases of
Theorem~\ref{thm:multiv} where $nm>1$ comprise the version for
``approximation of systems of linear forms.'' The subcases where $m=1$
and $n>1$ are often called the ``dual'' Duffin--Schaeffer conjecture.

\

For solutions to~(\ref{eq:1}) without the coprimality
requirement~(\ref{eq:coprimality}), we complete the following
result---also conjectured by Beresnevich, Bernik, Dodson, and Velani
in~\cite{BBDV}---by proving its $m=1, n>1$ cases. It is the
generalization of Catlin's conjecture to the setting of systems of
linear forms.

\begin{theorem}\label{thm:catlin}
  Let $m$ and $n$ be positive integers. For $\q\in\ZZ^n$, let
  \begin{equation*}
  \Phi_m(\q) = \#\set{\p\in\ZZ^m : \abs{\p} \leq \abs{\q},
    \gcd(\p,\q)=1}.
\end{equation*}
If $\psi:\ZZ^n\to\RR_{\geq 0}$ is a function such that
  \begin{equation}\label{eq:catlindiv}
    \sum_{\q\in\ZZ^n\setminus\set{0}}\Phi_m(\q) \sup_{t\geq 1}\parens*{\frac{\psi(t\q)}{t\abs{\q}}}^m = \infty,
  \end{equation}
  then for almost every $\x\in \operatorname{Mat}_{n\times m}(\RR)$
  there are infinitely many $(\p,\q)\in\ZZ^m\times\ZZ^n$ such
  that~(\ref{eq:1}) holds. Conversely, if the series converges, then
  there are almost no such $\x$.
\end{theorem}

When $m=n=1$, this is Catlin's conjecture~\cite[1976]{CatlinConj},
which follows from Koukoulopoulos and Maynard's proof of the
Duffin--Schaeffer conjecture. The cases of Theorem~\ref{thm:catlin}
where $m>1$ were proved by Beresnevich and
Velani~\cite{BVKG}. Interestingly, they do not follow from the $m>1$
cases of Theorem~\ref{thm:multiv}. We prove the remaining ``dual''
cases, where $m=1$ and $n>1$. In these cases, we show that
Theorem~\ref{thm:catlin} is actually equivalent to a version of
Theorem~\ref{thm:multiv} where the condition~(\ref{eq:coprimality})
has been removed.

\subsection{Context}
\label{sec:background}

These results belong to the field of metric Diophantine
approximation. The most basic questions in this area have to do with
approximating real numbers $x$ with rational numbers $p/q$. One
endeavors to understand how small the quantity $\abs{qx-p}$ can be,
while controlling the size of the denominator $q$. The seminal theorem
is Dirichlet's approximation theorem. It implies that for every real
$x$ there are infinitely many $(p,q)$ for which $\abs{qx-p}<1/q$.

In higher dimensions, one may want to approximate a vector
$\x\in\RR^m$ with rational vectors $\p/q$, where $\p\in\ZZ^m$ and
$q\in\NN$. This is referred to as ``simultaneous approximation''
because the components of $\x$ are real numbers which are being
approximated simultaneously by the components of $\p/q$, all having
the same denominator. The corresponding higher-dimensional Dirichlet
theorem says that for every $\x$ there are infinitely many $(\p,q)$
such that $\abs{q\x-\p} < q^{-1/m}$.

If we try to replace $q^{-1/m}$ by some other function $\psi(q)$, we
quickly find that in general only metric statements are possible. This
is the start of metric Diophantine approximation. The foundational
result here is due to
Khintchine~\cite{Khintchineonedimensional,Khintchine}.

\theoremstyle{plain}
\newtheorem*{theorem*}{Theorem}
\begin{theorem*}[Khintchine's theorem, 1926]
  If $\psi: \NN\to\RR_{\geq 0}$ is a non-increasing function such that
  \begin{equation}\label{eq:KTdiv}
    \sum_{q=1}^\infty \psi(q)^m = \infty,
  \end{equation}
  then for Lebesgue-almost every $\x\in\RR^m$ there are infinitely
  many $(\p,q)\in\ZZ\times\NN$ such that
  \begin{equation}\label{eq:5}
    \abs{q\x - \p}<\psi(q).
  \end{equation}
  Conversely, if the above series converges, then there are almost no
  $\x$ for which this holds.
\end{theorem*}

The ``convergence'' part of this theorem is a straightforward
consequence of the Borel--Cantelli lemma (indeed, all of the
``convergence'' parts in this paper are), and it would hold even
without the monotonicity assumption on $\psi$. This led to the
question of whether the monotonicity assumption was necessary for the
``divergence'' part. As it turns out, one can remove the assumption
when $m>1$ (Gallagher~\cite{Gallagherkt}), but not when $m=1$. Duffin
and Schaeffer showed this by constructing a counterexample function
$\psi$ which was supported on blocks of integers with high
multiplicative interdependence. In the same paper, they posed their
eponymous conjecture, where a coprimality assumption is imposed on the
approximating rationals and the divergence condition is suitably
modified~\cite{duffinschaeffer}. This avoids the issue of
multiplicative interdependence that arises in their counterexample.

It is worth explaining the previous paragraph in a bit more detail
here. The set of points $\x$ which is sought in all of the above
theorems is naturally seen as the ``limsup'' set for a sequence of
sets. For Khintchine's theorem, that sequence of sets is, for each
$q\in\NN$, the set of points $\x\in\RR^m$ which satisfy
$\abs{q\x-\p}<\psi(q)$ with some $\p\in\ZZ^m$. Since these sets are
periodic, it is enough to consider their intersections with $[0,1]^m$,
intersections which we denote $A(q)$ for the moment. The important
point here is that $\Leb(A(q)) \leq 2^m\psi(q)^m$, with equality if
$\psi(q) \leq 1/2$. That is, the sum in~(\ref{eq:KTdiv}) should be
viewed as a measure sum of the sets $A(q)$. This is why the
convergence part of Khintchine's theorem follows easily from the
Borel--Cantelli lemma. The reason the divergence part is more
difficult is that the sets $A(q)$ lack the probabilistic independence
needed to apply the second Borel--Cantelli lemma. The
Duffin--Schaeffer counterexample mentioned above exploits this issue
in dimension $m=1$, and the Duffin--Schaeffer conjecture avoids the
issue by requiring $\gcd(p,q)=1$. But with the coprimality requirement,
the limsup set in question corresponds to a sequence of smaller sets,
denoted $A'(q)$, whose measures are comparable to
$\frac{\varphi(q)\psi(q)}{q}$ in dimension one, and
$\parens*{\frac{\varphi(q)\psi(q)}{q}}^m$ in general. This is where
the divergence condition in the Duffin--Schaeffer conjecture comes
from, and by extension it is where~(\ref{eq:maindiv}) comes from.

\

The Duffin--Schaeffer conjecture animated a great deal of research in
Diophantine approximation until it was finally proved in a
breakthrough of Koukoulopoulos and Maynard in 2020~\cite{KMDS}. In the
meantime, many results were proved that brought us nearer to the full
conjecture:
\begin{itemize}
\item In~\cite{Gallagher01}, Gallagher established a zero-one law
  which showed that in order to prove the Duffin--Schaeffer
  conjecture, it sufficed to prove that its conclusion held for a
  positive measure set of real numbers $x$. 
\item In~\cite{ErdosDS}, Erd{\H o}s proved the Duffin--Schaeffer
  conjecture for functions taking only the values $0$ and $c/n$, for
  some fixed $c>0$.
\item In~\cite{Vaaler}, Vaaler proved the Duffin--Schaeffer conjecture
  under the additional assumption that $\psi(q) = O(1/q)$.
\item In~\cite{PollingtonVaughan,PVbordeaux}, Pollington and Vaughan
  proved the $m$-dimensional Duffin--Schaeffer conjecture, $m>1$.
\item
  In~\cite{HPVextra,BHHVextraii,AistleitnerDS,AistleitneretalExtraDivergence},
  Haynes--Pollington--Velani, Beresnevich--Harman--Haynes--Velani,
  Aistleitner, and
  Aistleitner--Lachmann--Munsch--Technau--Zafeiropoulos proved
  versions of the Duffin--Schaeffer conjecture where $\psi$ enjoys
  various forms of ``extra divergence.''
\end{itemize}
These hold an important place in metric Diophantine
approximation. Gallagher's zero-one law (and a higher-dimensional
version due to Vilchinski~\cite{Vilchinski}) was used in all of the
above cited results that followed it, including, of course, in the
eventual proof by Koukoulopoulos and Maynard. The overlap estimates
appearing in Erd{\H o}s's, Vaaler's, and Pollington \& Vaughan's work
(seen here as Proposition~\ref{lem:PVoverlaps}) contain in them the
obstacle that had to be overcome in order to prove the
Duffin--Schaeffer conjecture, hence the focus of most of the work that
followed. Naturally, these themes will appear in the present paper as
well.

\subsection{Approximation of systems of linear forms}
\label{sec:appr-syst-line}

During all of this, there has been the parallel development of the
theory of approximation of systems of linear forms. Here, one takes a
real $n\times m$ matrix $\x$ and examines the quantity
$\abs{\q\x - \p}$, with $(\p,\q)\in\ZZ^m\times\ZZ^n$. Groshev proved
an analogue of Khintchine's theorem for this kind of
approximation~\cite{Groshev}.

\begin{theorem*}[Khintchine--Groshev theorem, 1931]
  If $\psi: \NN\to\RR_{\geq 0}$ is a non-increasing function such that
  \begin{equation}\label{eq:KGdiv}
    \sum_{q=1}^\infty q^{n-1}\psi(q)^m = \infty,
  \end{equation}
  then for Lebesgue-almost every
  $\x\in\operatorname{Mat}_{n\times m}(\RR)$ there are infinitely many
  $(\p,\q)\in\ZZ^m\times\ZZ^n$ such that
  \begin{equation}\label{eq:6}
    \abs{\q\x - \p}<\psi(\abs{\q}).
  \end{equation}
  Conversely, if the above series converges, then there are almost no
  $\x$ for which this holds.
\end{theorem*}

As before,~(\ref{eq:KGdiv}) should be viewed as a measure sum. Here,
we can take the sets $A(q)$ to be all $\x\in[0,1]^{nm}$
satisfying~(\ref{eq:6}) with some $\q$ whose norm is $\abs{\q}=q$. The
measure of $A(q)$ is comparable to $q^{n-1}\psi(q)^m$ once 
$\psi(q)$ is small enough. Hence, the Borel--Cantelli lemma takes care
of the convergence part of the theorem, and one must work harder for
the divergence part.

Again, the question of whether monotonicity can be removed
arises. Sprindzuk showed in~\cite[Theorem~12]{Sprindzuk} that it can
be removed if $n>2$, leaving open the cases where $n=2$. (Recall that
the cases where $n=1$ and $m>1$ were handled by
Gallagher~\cite{Gallagher01} and that Duffin--Schaeffer showed that
monotonicity is required in the case $m=n=1$.) The question was
answered in full by Beresnevich and
Velani~\cite[Theorem~1]{BVKG}. They showed that one can remove the
monotonicity assumption from the Khintchine--Groshev theorem whenever
$nm>1$.

\subsection{Multivariate theory}

Now that $\q$ ranges through $\ZZ^n$, there is a more general setup
that one can consider. Namely, one should allow for $\psi$ to be
multivariate, meaning it depends on $\q$ rather than just
$\abs{\q}$. In this setting, there is the following theorem of
Beresnevich and Velani~\cite[Theorem~5]{BVKG}, which we view as a
multivariate analogue of Gallagher's version of Khintchine's theorem
without monotonicity~\cite{Gallagher01}.

\begin{theorem*}[Beresnevich--Velani, 2010]
  Let $m>1$. If $\psi: \ZZ^n\to\RR_{\geq 0}$ is a function such that
  \begin{equation}\label{eq:BVdiv}
    \sum_{\q\in\ZZ^n} \psi(\q)^m = \infty,
  \end{equation}
  then for Lebesgue-almost every
  $\x\in\operatorname{Mat}_{n\times m}(\RR)$ there are infinitely many
  $(\p,\q)\in\ZZ^m\times\ZZ^n$ such that~(\ref{eq:1}) holds. 
\end{theorem*}

\begin{remark*}
  This improved on another result of
  Sprindzuk~\cite[Theorem~14]{Sprindzuk}, where the same was proved
  for functions $\psi$ supported on the primitive vectors of $\ZZ^n$.
\end{remark*}

As with the other theorems,~(\ref{eq:BVdiv}) should be viewed as a
measure sum. Here, the relevant sets $A(\q)$ consist of
$\x\in[0,1]^{nm}$ satisfying~(\ref{eq:1}).

The above theorem is equivalent to the $m>1$ cases of
Theorem~\ref{thm:catlin} (see~\cite{BBDV}). Indeed, when $m>1$ we have
that $\Phi_m(\q)$ is comparable to $\abs{\q}^m$. Armed with this fact,
one can quickly show that Beresnevich and Velani's theorem implies the
$m>1$ cases of Theorem~\ref{thm:catlin}. (Similarly, Gallagher's
non-monotonic Khintchine theorem implies a higher-dimensional
``simultaneous'' Catlin conjecture.) However, an example like the one
constructed by Duffin and Schaeffer shows that Beresnevich and
Velani's theorem cannot possibly hold for $m=1$, and so the dual
Catlin conjecture cannot be handled in the same way. We prove it here
as a corollary to the dual Duffin--Schaeffer conjecture
(Theorem~\ref{thm:multiv} with $m=1$), whose coprimality
requirement~(\ref{eq:coprimality}) avoids the issue caused by the
(dual) Duffin--Schaeffer counterexample.

Regarding the Duffin--Schaeffer conjecture for systems of linear forms
(the full Theorem~\ref{thm:multiv}), the result is new in all cases
where $n>1$. Its relationship to the $n>1$ cases of the
Khintchine--Groshev theorem, the generalized version of Catlin's
conjecture (Theorem~\ref{thm:catlin}), and Beresnevich \& Velani's
theorem is the same as the relationship between the $n=1$ cases of
Theorem~\ref{thm:multiv} (the classical Duffin--Schaeffer conjecture
and its $m$-dimensional analogue) and the $n=1$ versions of the
aformentioned theorems (specifically, Khintchine's theorem, Catlin's
conjecture and its ``simultaneous'' analogues, and Gallagher's
non-monotonic Khintchine theorem), in terms of the implications and
non-implications between them.

\subsection{Idea of the proof}
\label{sec:idea-proof}

The general goal of the proof of Theorem~\ref{thm:multiv} is one that
pervades this field. In order to show that a limsup set is large, one
must show that the corresponding sequence of sets does not overlap
itself too much. That is, we must ensure that the measure
sum~(\ref{eq:maindiv}) is not overly redundant. In this work, we will
find that the geometric properties of our approximation sets allow us
to control many of these overlaps, and for the ones that cause
difficulties, we reframe the problem in a way that allows us to apply
overlap estimates that were achieved in the work of Koukoulopoulos \&
Maynard~\cite{KMDS} and Pollington \&
Vaughan~\cite{PollingtonVaughan,PVbordeaux}.

To say a bit more, recall that the second Borel--Cantelli lemma offers
us one way to turn control of overlaps into a full-measure statement:
For a sequence of pairwise independent sets having diverging measure
sum, the limsup set has full measure. As it turns out, this is not
directly helpful for proving the kinds of theorems we are after,
because the sets involved are not pairwise independent. For example,
the sets $A(q)$ appearing above in our short discussion of
Khintchine's theorem are not pairwise independent. But in the case of
Khintchine's theorem, the fact that $\psi$ is monotonic allows one to
show that they do enjoy a type of ``average'' independence, where
instead of $\Leb(A(q)\cap A(r))$ being bounded by (a constant times)
$\Leb(A(q))\Leb(A(r))$, one has the bound on average over large sets
of pairs $(q,r)$. This turns out to be enough to prove Khintchine's
theorem, after application of a zero-one law.

For the sets $A_{1,m}'(q)$ appearing in the $m$-dimensional
Duffin--Schaeffer conjecture, the situation is harder, in part because
$\psi$ is arbitrary. Nevertheless, Pollington and Vaughan achieved an
averaged pairwise independence in the $m>1$ cases and, in a momentous
breakthrough, Koukoulopoulos and Maynard managed it in the $m=1$
setting. In the end, our proof depends crucially on those
accomplishments. The specific parts of them that we avail ourselves of
are encapsulated in Section~\ref{sec:pollingtonvaughan},
Proposition~\ref{prop:KMVP}.

In the setting of systems of linear forms, we denote the relevant sets
by $A_{n,m}'(\q)$ (to be formally introduced in
Section~\ref{sec:notation}). We pursue the same goal of showing that
they enjoy this kind of averaged independence. The sets $A_{n,m}'(\q)$
are different in shape from the sets $A_{1,m}'(q)$. The latter are
unions of metric balls, whereas if $n>1$ then $A_{n,m}'(\q)$ is a
union of ``stripes'' through $[0,1]^{nm}$. One effect of this is that
when $\q$ and $\br$ are linearly independent, the sets $A_{n,m}'(\q)$
and $A_{n,m}'(\br)$ are independent (Section~\ref{sec:lemmata},
Lemma~\ref{lem:indA}). This means that when computing the average
measures of intersections, the contributions from linearly independent
pairs $(\q,\br)$ are manageable, and we only have to control the
contributions from dependent pairs. The idea for dependent pairs is to
show that for $1\leq k< \ell$, the pairwise intersections between sets
from the collection $\set*{A_{n,m}'(k\q)}_{\gcd(\q)=1}$ and sets from
the collection $\set*{A_{n,m}'(\ell\q)}_{\gcd(\q)=1}$ have a
cumulative contribution to the average which is comparable to the
contribution of $A_{1,m}'(k)\cap A_{1,m}'(\ell)$ in the $(1,m)$
setting, with an approximating function $\Psi:\NN\to\RR_{\geq 0}$
defined by $\Psi(d) = \parens*{\sum_{\gcd(\q)=1}\psi(d\q)^m}^{1/m}$
(Section~\ref{sec:proof}, Proof of Theorem~\ref{thm:multiv}). At this
point we are able to apply Proposition~\ref{prop:KMVP}, the
aforementioned results of Koukoulopoulos \& Maynard and Pollington \&
Vaughan. Once all contributions are collected, we find the averaged
pairwise independence that is needed to guarantee that the limsup set
has positive measure. The theorem follows after applying an analogue
of Gallagher's zero-one law due to Beresnevich and Velani
(Section~\ref{sec:pollingtonvaughan}, Theorem~\ref{thm:01}).

\subsection{Hausdorff measure analogues}
\label{sec:hausd-meas-anal}

In~\cite{BVmassTP} Beresnevich and Velani proved the Mass Transference
Principle, a result that allows us to deduce Hausdorff measure
statements about limsup sets for sequences of metric balls from
Lebesgue measure statements for the same. Through it, they showed that
the Duffin--Schaeffer conjecture implies its own Hausdorff measure
analogues. The Mass Transference Principle has wide application. Using
it in combination with Dirichlet's theorem, one can prove the
Jarnik--Besicovitch theorem, and, in combination with Khitchine'
theorem, one obtains Jarnik's theorem.

In~\cite{AllenBeresnevich} Allen and Beresnevich modified the Mass
Transference Principle in a way that allows it to be applied to limsup
sets of sequences of thickened affine subspaces,
confirming~\cite[Conjecture~E]{BBDV}. Using that in combination with
the Khintchine--Groshev theorem, they proved a Hausdorff measure
analogue of Khintchine--Groshev~\cite[Theorem~2]{AllenBeresnevich},
partial versions of which had previously been obtained through other
methods~\cite{BBDV, DickinsonVelaniCrelle}. In~\cite{BBDV}, the
authors point out that their Conjecture~E together with the
Duffin--Schaeffer conjecture for systems of linear forms would yield
Hausdorff measure analogues, and similarly for Catlin's conjecture for
systems of linear forms. Specifically, Theorems~\ref{thm:multiv}
and~\ref{thm:catlin} in this paper, together with the Mass
Transference Principle as extended in~\cite{AllenBeresnevich}, give
the following.

\begin{theorem}[General Duffin--Schaeffer,~{\cite[Conjecture~F]{BBDV}}]
  Let $f$ and $g:r\mapsto g(r):= r^{-m(n-1)}f(r)$ be dimension
  functions such that $r^{-mn}f(r)$ is monotonic. Let
  $\psi:\ZZ^n\to\RR_{\geq 0}$ be a function. Then
  \begin{equation*}
    \mathcal H^f(W'(\psi)) = \mathcal H^f([0,1]^{nm}) \quad\textrm{if}\quad \sum_{\q\in\ZZ^n\setminus\set{0}} \parens*{\frac{\varphi(\gcd(\q))}{\gcd(\q)}\abs{\q}}^m \cdot g\parens*{\frac{\psi(\q)}{\abs{\q}}}= \infty.
  \end{equation*}
\end{theorem}

\begin{theorem}[General Catlin,~{\cite[Conjecture~G]{BBDV}}]
  Let $f$ and $g:r\mapsto g(r):= r^{-m(n-1)}f(r)$ be dimension
  functions such that $r^{-mn}f(r)$ is monotonic. Let
  $\psi:\ZZ^n\to\RR_{\geq 0}$ bea function. Then
  \begin{equation*}
    \mathcal H^f(W(\psi)) = \mathcal H^f([0,1]^{nm}) \quad\textrm{if}\quad \sum_{\q\in\ZZ^n\setminus\set{0}} \Phi_m(\q)\cdot \sup_{t\in\NN} g\parens*{\frac{\psi(t\q)}{t\abs{\q}}} = \infty,
  \end{equation*}
  with $\Phi_m$ defined as in Theorem~\ref{thm:catlin}.
\end{theorem}

We refer the reader to~\cite{BVmassTP, AllenBeresnevich, BBDV} for
full discussions of the Mass Transference Principle and its
applications in Diophantine approximation.

\subsection{Contents of the paper}

In Section~\ref{sec:notation} we collect notation, conventions, and
initial reductions. Section~\ref{sec:pollingtonvaughan} contains
results of other authors that are crucial in this and other works on
the topic. In order to streamline the presentation of the proof of
Theorem~\ref{thm:multiv}, we collect several useful lemmas in
Section~\ref{sec:lemmata}. Section~\ref{sec:proof} is the proof of
Theorem~\ref{thm:multiv}. In Section~\ref{sec:proof-theor-refthm:c} we
show how the $m=1$ cases of Theorem~\ref{thm:catlin} follow from
Theorem~\ref{thm:multiv}.

\subsection*{Acknowledgments}
\label{sec:acknowledgements}

I thank Victor Beresnevich and Dimitris Koukoulopoulos for helpful
conversations and advice.

\section{Notation, conventions, and initial reductions}
\label{sec:notation}

Notice that if the divergence condition in Theorem~\ref{thm:multiv} is
met, then there is some orthant in $\ZZ^n$ where it is met, and it
suffices to prove the theorem for a version of $\psi$ which is
supported in that orthant. Moreover, by applying reflections over
coordinate hyperplanes, we may assume without loss of generality that
it is the positive orthant, so that it suffices to prove
Theorem~\ref{thm:multiv} for functions of the form
$\psi:\ZZ_{\geq 0}^n\to \RR_{\geq 0}$.

Since the set of $\x\in\operatorname{Mat}_{n\times m}(\RR)$ for which
there are infinitely many solutions to~(\ref{eq:1}) is periodic, it is
customary to restrict attention to those $\x$ lying in
$[0,1]^{nm}$. The set of such $\x$ is denoted $W(\psi)$. We are
further interested in a subset of $W(\psi)$, namely those
$\x\in[0,1]^{nm}$ for which~(\ref{eq:1}) has infinitely many solutions
which additionally satisfy the coprimality
condition~(\ref{eq:coprimality}). That subset is denoted $W'(\psi)$,
and the conclusion of Theorem~\ref{thm:multiv} is equivalent to
\begin{equation}
  \label{eq:measureone}
  \Leb\parens*{W'(\psi)} = 1.
\end{equation}
This is what we will show.

It is useful to identify $[0,1]^{nm}$ with the $nm$-dimensional torus
$\TT^{nm} = \RR^{nm}/\ZZ^{nm}$. We view its elements as $n\times m$
matrices. For $\eps>0$ and $\q\in\ZZ_{\geq 0}^n$, let
\begin{equation*}
  A_{n,m}'(\q, \eps) = \set*{\x\in\TT^{nm} : \exists \p\in\ZZ^m, \gcd(p_i, \q)=1\, (i=1, \dots, m),\, \abs{\q\x-\p}<\eps}. 
\end{equation*}
Notice that
\begin{equation*}
  \Leb\parens*{A_{n,m}'(\q, \eps)} \leq \parens*{\frac{2\varphi(\gcd(\q))\eps}{\gcd(\q)}}^m.
\end{equation*}
For a function $\psi:\ZZ_{\geq 0}^n\to\RR_{\geq 0}$, denote
\begin{equation*}
  A_{n,m}'(\q) := A_{n,m}'(\q, \psi(\q)) = \set*{\x\in\TT^{nm} : \exists \p\in\ZZ^m  \textrm{ such that~(\ref{eq:1}) and~(\ref{eq:coprimality}) hold}}. 
\end{equation*}
This way,
\begin{equation*}
W'(\psi) = \limsup_{q\to\infty}\bigcup_{\abs{\q} = q}A_{n,m}'(\q)
\end{equation*}
is the set of points lying in infinitely many of the sets
$A_{n,m}'(\q)$, and the sum appearing in Theorem~\ref{thm:multiv}
bounds (half) the sum of their measures.

Let $\PP^n$ denote the (nonzero) primitive vectors in $\ZZ_{\geq 0}^n$. Then the
divergence condition in Theorem~\ref{thm:multiv} can be rephrased
as
\begin{equation}\label{eq:div}
  \sum_{\q\in\PP^n}\sum_{d=1}^\infty \parens*{\frac{\varphi(d)}{d}\psi(d\q)}^m = \infty.
\end{equation}
With this, Theorem~\ref{thm:multiv} can be stated as
\begin{equation*}
  (\ref{eq:div}) \quad\implies\quad (\ref{eq:measureone}). 
\end{equation*}
For $d\geq 1$, let
\begin{equation*}
  \Psi(d) = \parens*{\sum_{\q\in\PP^n} \psi(d\q)^m}^{1/m}.
\end{equation*}
Eventually, we will see that no generality is lost in assuming that
$\Psi(d)$ is always finite.

Finally, we use the Vinogradov asymptotic notation throughout. For
two functions $f(t), g(t)$, the notations $f(t) \ll g(t)$ and
$g(t)\gg f(t)$ are used to mean that there is some constant $c>0$ such
that $f(t) \leq cg(t)$ for all relevant values of $t$. The notation
$f(t)\asymp g(t)$ is used to mean that $f(t)\ll g(t)$ and
$f(t)\gg g(t)$ both hold.  The implicit constants $c>0$ appearing here
depend only on the parameter $m$.

\section{Essential ingredients}
\label{sec:pollingtonvaughan}

The challenging parts of all of the theorems mentioned in
Section~\ref{sec:intro} are the divergence parts. As we have
mentioned, this is because the sets we work with are not pairwise
independent, such as one would need in order to apply the second
Borel--Cantelli lemma to achieve full measure of the limsup set. The
following overlap estimates have appeared
in~\cite{Strauch,ErdosDS,Vaaler,PollingtonVaughan,PVbordeaux,AistleitneretalExtraDivergence,AistleitnerDS,BHHVextraii,HPVextra,KMDS}. In
a sense, they represent the best one can do when it comes to pairwise
independence.

\begin{proposition}[Overlap estimates]\label{lem:PVoverlaps}
  Let $\psi:\NN\to \RR_{\geq 0}$ be such that
  $\Psi(d) \leq  d/(2\varphi(d))$ for every $d\geq 1$. If
  $k\neq\ell$, then
  \begin{equation*}
    \Leb\parens*{A_{1,1}'(k, \Psi(k))\cap A_{1,1}'(\ell, \Psi(\ell))} \ll \bone_{M(k,\ell)\geq \gcd(k,\ell)}\frac{\varphi(k)\Psi(k)}{k}\frac{\varphi(\ell)\Psi(\ell)}{\ell}\prod_{\substack{p\mid k\ell/\gcd(k,\ell)^2 \\ p > M(k,\ell)/\gcd(k,\ell)}}\parens*{1 + \frac{1}{p}},
  \end{equation*}
  where $M(k,\ell) = \max\set*{\ell\Psi(k), k\Psi(\ell)}$. The
  implicit constant is absolute.
\end{proposition}

\begin{proof}
  This is found in~\cite[Pages195--196]{PollingtonVaughan} as well
  as~\cite[Lemma~2.8]{Harman}.
\end{proof}

Much of the effort in the above cited papers goes into controlling the
effect that the factor
\begin{equation*}
  \prod_{\substack{p\mid k\ell/\gcd(k,\ell)^2 \\ p > M(k,\ell)/\gcd(k,\ell)}}\parens*{1 + \frac{1}{p}}
\end{equation*}
can have on the overlaps. In~\cite{PollingtonVaughan,PVbordeaux,KMDS},
Pollington \& Vaughan and Koukoulopoulos \& Maynard manage to control
this effect on average. The following proposition summarizes the
control they achieve.

\begin{proposition}[Koukoulopoulos--Maynard, Pollington--Vaughan]\label{prop:KMVP}
  Let $\Psi:\NN\to [0,1/2]$ be a function. Then for all large $X\geq 1$ and $Y>X$ such that
  \begin{equation*}
    \parens*{\frac{1}{2}}^{m-1} < \sum_{X\leq d \leq Y}\parens*{\frac{\varphi(d)\Psi(d)}{d}}^m < \parens*{\frac{1}{2}}^{m-2},
  \end{equation*}
  one has
  \begin{equation*}
    \sum_{X\leq k < \ell \leq Y}  \Leb\parens*{A_{1,m}'(k,
      \Psi(k))\cap A_{1,m}'(\ell, \Psi(\ell))} \ll 1.
  \end{equation*}
  The implicit constant depends only on $m$.
\end{proposition}

\begin{proof}
  In the case $m=1$, this is contained in~\cite[Proof of Theorem~1 assuming
  Proposition~5.4]{KMDS}.

  In the cases where $m>1$, it is contained in~\cite[Pages~197--199,
  Cases~(i)--(iii)]{PollingtonVaughan}, and
  in~\cite[Cases~(a),(b),(d)]{PVbordeaux}, but the notation is
  slightly different, so some translation is in order. Pollington and
  Vaughan let $\eta\in(0,1)$ be any fixed parameter, and
  $Z\subset \NN$ a finite set (with large minimum) such that
  \begin{equation*}
    \frac{1}{2}\eta < \sum_{d\in Z}\Leb\parens*{A_{1,m}(d, \Psi(d))} < \eta.
  \end{equation*}
  They show
  \begin{equation*}
    \sum_{ \substack{k,\ell \in Z\\ k < \ell}}  \Leb\parens*{A_{1,m}'(k,
      \Psi(k))\cap A_{1,m}'(\ell, \Psi(\ell))} \ll \sum_{d\in Z}\Leb\parens*{A_{1,m}'(d, \Psi(d))},
  \end{equation*}
  with an implicit constant depending at most on $m$. In our
  proposition's statement, we have taken $\eta = 2^{2-m}$ and the
  range $Z = [X,Y]\cap \ZZ$.
\end{proof}

With this on-average bound of overlaps, Pollington \& Vaughan and
Koukoulopoulos \& Maynard are able to show, through an application of
the Chung--Erd{\H o}s lemma (a version of which is listed below as
Proposition~\ref{lem:erdoschung}), that the associated limsup sets
have positive measure. At this point, Gallagher's zero-one
law~\cite{Gallagher01} and Vilchinski's higher-dimensional
version~\cite{Vilchinski} lead to the proof of the Duffin--Schaeffer
conjecture in dimensions $m\geq 1$.

\

Here is the version of the Chung--Erd{\H o}s lemma that we will
use. Usually, it is stated for sums of finitely many measurable sets,
but it is not hard to show that it also holds for infinitely many sets
having a finite measure sum.

\begin{proposition}[Chung--Erd\H{o}s Lemma]\label{lem:erdoschung}
  If $(X, \mu)$ is a probability space and
  $(A_q)_{q \in \NN}\subset X$ is a sequence of measurable subsets
  such that $0 < \sum_{q} \mu(A_q) < \infty$, then
  \begin{equation*}
    \mu\parens*{\bigcup_{q=1}^\infty  A_q} \geq \frac{\parens*{\sum_{q=1}^\infty \mu(A_q)}^2}{\sum_{q,r=1}^\infty \mu(A_q\cap A_r)}.
  \end{equation*}
\end{proposition}

\begin{proof}
  For any $Q>1$, let
  \begin{equation*}
    C_Q = \bigcup_{q=1}^Q A_q.
  \end{equation*}
  Applying the Cauchy--Schwarz inequality, we have
  \begin{align*}
    \parens*{\int_X \sum_{q=1}^Q \bone_{A_q}\,d\mu}^2
    &=\parens*{\int_X \bone_{C_Q}\sum_{q=1}^Q \bone_{A_q}\,d\mu}^2 \\
    &\leq \mu(C_Q) \int_X \parens*{\sum_{q=1}^Q \bone_{A_q}}^2\,d\mu \\
    &= \mu(C_Q)\sum_{q,r=1}^Q \mu(A_q\cap A_r).
  \end{align*}
  Now
  \begin{align*}
    \mu\parens*{\bigcup_{q=1}^\infty A_q}
    &= \lim_{Q\to\infty}\mu(C_Q) \\
    &\geq \limsup_{Q\to\infty}\frac{\parens*{\sum_{q=1}^Q \mu(A_q)}^2}{\sum_{q,r=1}^Q \mu(A_q\cap A_r)}.
  \end{align*}
  For any $\eps>0$, we have
  \begin{align*}
    \parens*{\sum_{q=1}^Q \mu(A_q)}^2 \geq (1-\eps)\parens*{\sum_{q=1}^\infty \mu(A_q)}^2
  \end{align*}
  for all sufficienly large $Q$. Therefore, 
  \begin{align*}
\limsup_{Q\to\infty}\frac{\parens*{\sum_{q=1}^Q \mu(A_q)}^2}{\sum_{q,r=1}^Q \mu(A_q\cap A_r)} \geq (1-\eps)\frac{\parens*{\sum_{q=1}^\infty \mu(A_q)}^2}{\sum_{q,r=1}^\infty \mu(A_q\cap A_r)},
  \end{align*}
  and the result follows since $\eps>0$ was arbitrary.
\end{proof}

Finally, we will need a zero-one law that holds in the setting of
approximation of systems of linear forms. This is provided by
Beresnevich and Velani~\cite[Theorem~1]{BVzeroone}.

\begin{theorem}[Beresnevich--Velani]\label{thm:01}
  For any $n, m$, and $\psi$ we have $\Leb(W'(\psi))\in\set{0,1}$.
\end{theorem}

\section{Lemmas}
\label{sec:lemmata}

This section is a collection of the lemmas we will need for the proof
of Theorem~\ref{thm:multiv}.

\begin{lemma}\label{lem:indE}
  For $\eps>0$, $\q\in\ZZ^n$, and $\bv\in\RR^n$, let
  \begin{equation*}
    E(\q,\bv,\eps) = \pi\parens*{\set*{\x\in \RR^n : \abs*{\q\cdot\x} < \eps} + \bv},
  \end{equation*}
  where $\pi:\RR^n \to \TT^n$ is the natural quotient map. If
  $\q_1,\q_2\in\ZZ^n$ are linearly independent, then
  \begin{equation*}
    \Leb\parens*{E(\q_1,\bv_1,\eps_1)\cap E(\q_2,\bv_2, \eps_2)} = \Leb\parens*{E(\q_1,\bv_1, \eps_1)}\Leb\parens*{E(\q_2,\bv_2, \eps_2)},
  \end{equation*}
  for all $\eps_1,\eps_2>0$ and $\bv_1, \bv_2\in\RR^n$.
\end{lemma}

\begin{proof}
  First, note that if $\eps_1\geq 1/2$ or $\eps_2\geq 1/2$, then at
  least one of $E(\q_1,\bv_1,\eps_1)$ and $E(\q_2,\bv_2,\eps_2)$ is
  all of $\TT^n$ and the lemma will clearly hold. So we will assume
  that $0< \eps_1, \eps_2 < 1/2$.

  Next, note that the vectors $\bv_1, \bv_2$ do not increase the
  generality, because the set
  \begin{equation*}
    E(\q_1,\bv_1,\eps_1)\cap E(\q_2,\bv_2, \eps_2)
  \end{equation*}
  is a translate of
  \begin{equation*}
    E(\q_1,0 ,\eps_1)\cap E(\q_2,0 , \eps_2)
  \end{equation*}
  by a solution $\x_0$ to $\q_1\cdot(\x_0 - \bv_1) = \q_2\cdot(\x_0 - \bv_2)$. Such
  a solution $\x_0$ is guaranteed to exist because $\q_1$ and $\q_2$
  are linearly independent. So it is enough to treat the case where
  $\bv_1=\bv_2=0$.

  Let $\bone_i$ denote the indicator function of the interval
  $(-\eps_i, \eps_i)$. Then
  \begin{equation*}
    \Leb(E_1\cap E_2)
    = \int_{\TT^n} \bone_1(\q_1\cdot \x)\bone_2(\q_2\cdot \x)\,d\x.
  \end{equation*}
  Expanding $\bone_i$ into a Fourier series
  \begin{equation*}
    \bone_i (x) \sim \sum_{n\in\ZZ} a_i(n) e(nx), 
  \end{equation*}
  where $e(r):=\exp(2\pi i r)$, we have 
  \begin{align*}
    \Leb(E_1\cap E_2)
    &= \int_{\TT^n} \parens*{\sum_{n\in\ZZ} a_1(n) e(n \q_1\cdot\x)}\parens*{\sum_{n\in\ZZ} a_2(n) e(n \q_2\cdot\x)}\,d\x \\
    &= \sum_{n_1, n_2\in\ZZ}  a_1(n_1) a_2(n_2) \underbrace{\int_{\TT^n}e((n_1 \q_1+ n_2\q_2)\cdot\x)\,d\x.}_{=0 \textrm{ if } n_1\q_1+n_2\q_2\neq 0}\\ 
    &= a_1(0) a_2(0)\\
    &= \Leb(E_1)\Leb(E_2). 
  \end{align*}
  This is what we had to show.
\end{proof}

\begin{lemma}\label{lem:indA}
  If $\q_1,\q_2\in\ZZ^n$ are linearly independent, then
  \begin{equation*}
    \Leb\parens*{A_{n,m}'(\q_1)\cap A_{n,m}'(\q_2)} = \Leb\parens*{A_{n,m}'(\q_1)}\Leb\parens*{A_{n,m}'(\q_2)},
  \end{equation*}
  that is, the sets $A_{n,m}'(\q_1)$ and $A_{n,m}'(\q_2)$ are independent.
\end{lemma}

\begin{proof}
  The lemma clearly holds if either of the sets $A_{n,m}'(\q_i)$ has measure
  $0$. Let us assume that they both have positive measure. The set
  $A_{n,m}'(\q_i)$ $(i=1,2)$ has the form
  \begin{equation*}
    A_{n,m}'(\q_i) = \prod_{j=1}^m A_{n,1}'(\q_i)
  \end{equation*}
  where the $j$th factor corresponds to the $j$th column in $\TT^{nm}$
  viewed as $n\times m$ matrices. Therefore, 
  \begin{equation}\label{eq:powerm}
    \Leb\parens*{A_{n,m}'(\q_i)} = \Leb\parens*{A_{n,1}'(\q_i)}^m
  \end{equation}
  and
  \begin{equation}\label{eq:powermint}
    \Leb\parens*{A_{n,m}'(\q_1)\cap A_{n,m}'(\q_2)} = \Leb\parens*{A_{n,1}'(\q_1)\cap A_{n,1}'(\q_2)}^m
  \end{equation}
  where $\Leb$ is to be understood as Lebesgue measure in the
  appropriate dimension.

  We claim that each $A_{n,1}'$ is a disjoint union of
  finitely many sets of the form $E(\q_i,\bv, \eps)$. To see it, note
  that we have
  \begin{equation*}
    A_{n,1}'(\q) = \bigcup_{\abs{p}\leq\abs{\q}} E(\q, \bv_{p,\q}, \psi(\q)),
  \end{equation*}
  where $\bv_{p,\q}$ is any solution to $\q\cdot \bv = p$. The sets
  $E(\q, \bv_{p,\q}, \psi(\q))$ may not be disjoint, but any two that
  intersect have a union that is of the form $E(\q,\bv,\eps)$ for some
  other $\bv$ and $\eps$. The claim follows upon recombining.

  Therefore, by Lemma~\ref{lem:indE} we have
  \begin{equation*}
    \Leb\parens*{A_{n,1}'(\q_1)\cap A_{n,1}'(\q_2)} = \Leb\parens*{A_{n,1}'(\q_1)}\Leb\parens*{A_{n,1}'(\q_2)}.
  \end{equation*}
  The lemma is proved by combining this last expression
  with~(\ref{eq:powerm}) and~(\ref{eq:powermint}).
\end{proof}

\begin{lemma}\label{lem:projection}
  For $d,k,\ell\geq 1$ and $\q\in\PP^n$, we have
  \begin{equation*}
    \Leb\parens*{A_{n,m}'(d\q)} = \Leb\parens*{A_{1,m}' \parens*{d, \psi(d\q)}}
  \end{equation*}
  and
  \begin{equation*}
    \Leb\parens*{A_{n,m}'(k\q)\cap A_{n,m}'(\ell\q)}  = \Leb\parens*{A_{1,m}'(k, \psi(k\q))\cap A_{1,m}'(\ell, \psi(\ell\q))}.
  \end{equation*}
\end{lemma}

\begin{proof}
  Because of~(\ref{eq:powerm}) and~(\ref{eq:powermint}), it is
  sufficient to prove the lemma in the case $(n,m) = (n,1)$. We have
  \begin{equation*}
    A_{n,1}'(d\q) = \set*{\x\in\TT^n : \exists p\in\ZZ,\, \abs*{d\q\cdot\x - p} < \psi(d\q),\, \gcd(p,d)=1}. 
  \end{equation*}
  Let $T_\q: \TT^n\to \TT^1$ be the transformation $\x\mapsto \q\cdot\x\,(\bmod\, 1)$, and notice that
  \begin{equation*}
    A_{n,1}'(d\q) = T_\q^{-1}\parens*{A_{1,1}'(d, \psi(d\q))}. 
  \end{equation*}
  Since $T_\q$ preserves Lebesgue measure, the lemma follows.
\end{proof}

\begin{lemma}\label{lem:dilation}
  Suppose $I_1, I_2,\dots, I_r \subset \TT^m$ are disjoint balls and
  $J_1, J_2,\dots, J_r\subset \TT^m$ are disjoint balls. Then for any
  $0 < \Sigma \leq 1$,
  \begin{equation*}
    \Leb\parens*{\bigcup_{i=1}^r\Sigma \bullet I_i \cap \bigcup_{j=1}^s\Sigma\bullet J_j} \leq \Sigma^m \Leb\parens*{\bigcup_{i=1}^r I_i\cap\bigcup_{j=1}^s J_j},
  \end{equation*}
  where $\Sigma \bullet I_i$ denotes the concentric scaling of the ball
  $I_i$ by $\Sigma$, and similar for $\Sigma \bullet J_j$.
\end{lemma}

\begin{proof}
  By the disjointness assumption, we have
  \begin{equation}\label{eq:7}
    \Leb\parens*{\bigcup_{i=1}^r I_i\cap\bigcup_{j=1}^s J_j} = \sum_{i,j}\Leb\parens*{I_i\cap J_j}
  \end{equation}
  and
  \begin{equation}\label{eq:8}
    \Leb\parens*{\bigcup_{i=1}^r \Sigma\bullet I_i\cap\bigcup_{j=1}^s \Sigma\bullet J_j} = \sum_{i,j}\Leb\parens*{\Sigma\bullet I_i\cap\Sigma\bullet J_j}
  \end{equation}
  Clearly,
  \begin{equation}\label{eq:9}
    \Leb\parens*{\Sigma I_i\cap \Sigma J_j} = \Sigma^m \Leb\parens*{I_i\cap J_j},
  \end{equation}
  where $\Sigma I_i$ and $\Sigma J_j$ denote the images of $I_i$ and
  $J_j$ under the contraction of $[0,1]^m\subset \RR^m$ by $\Sigma$. Since
  \begin{equation*}
    \Leb(\Sigma I_i) = \Leb(\Sigma\bullet I_i)\quad\textrm{and}\quad\Leb(\Sigma J_j) = \Leb(\Sigma\bullet J_j),
  \end{equation*}
  and the centers of $\Sigma\bullet I_i$ and $\Sigma\bullet J_j$ are
  farther apart that the centers of $\Sigma I_i$ and $\Sigma J_j$, we
  have
  \begin{equation}\label{eq:10}
    \Leb\parens*{\Sigma\bullet I_i\cap\Sigma\bullet J_j} \leq \Leb\parens*{\Sigma I_i\cap\Sigma J_j}.
  \end{equation}
  The lemma follows on
  combining~(\ref{eq:7}),~(\ref{eq:8}),~(\ref{eq:9}),
  and~(\ref{eq:10}).
\end{proof}

\begin{lemma}\label{lem:size}
  Let $d\geq 1, \q\in\PP^n$. If $\psi(d\q) \leq d/(2\varphi(d))$, then
  \begin{equation*}
    \Leb\parens*{A_{n,m}'(d\q)} \geq \parens*{\frac{\varphi(d)}{d}\psi(d\q)}^m.
  \end{equation*}
 In fact, for every $d$, there is a subset
  \begin{equation*}
    P_d \subset \set*{\p \in \ZZ^m : \abs{\p}\leq d, \gcd(p_i, d)=1, (i=1\dots, m)}
  \end{equation*}
  such that $\#P_d \geq 3^{-m}\varphi(d)^m$ and such that 
  \begin{equation}\label{eq:12}
  \abs*{\frac{\p_1}{d} - \frac{\p_2}{d}}\geq \frac{d}{\varphi(d)}
\end{equation}
for any two distinct $\p_1, \p_2\in P_d$.
\end{lemma}

\begin{proof}
  The first part is~\cite[Page 192, Equation (3)]{PollingtonVaughan}
  applied to $A_{1,1}' \parens*{d, \psi(d\q)}$, which has the same
  measure as $A_{n,1}'(d\q)$, by Lemma~\ref{lem:projection}. We then
  apply~(\ref{eq:powerm}).

  The second part can be deduced from the first part in the following
  way. We have
  \begin{equation*}
    \Leb\parens*{A_{1,m}'\parens*{d, \frac{d}{2\varphi(d)}}} \geq \frac{1}{2^m}.
  \end{equation*}
  By Vitali's covering lemma, there is a disjoint subcollection
  \begin{equation*}
    A''\subset A_{1,m}'\parens*{d, \frac{d}{2\varphi(d)}}
  \end{equation*}
  of balls such that the union of their $3$-fold dilates covers
  $A_{1,m}'\parens*{d, \frac{d}{2\varphi(d)}}$. Denote by $P_d$ the set of
  numerator vectors of the centers of the balls making up $A''$. Then
  \begin{equation*}
    \Leb\parens*{A_{1,m}'\parens*{d, \frac{d}{2\varphi(d)}}} \leq \# P_d\cdot\parens*{\frac{3}{2\varphi(d)}}^m.
  \end{equation*}
  Combining, we have $\# P_d \geq 3^{-m} \varphi(d)^m$. And since the
  disjoint balls comprising $A''$ have radius $d/(2\varphi(d))$, we
  have~(\ref{eq:12}).
\end{proof}

\begin{lemma}\label{lem:doverphid}
  Suppose that for every $d\geq 1$ and $\q\in\PP^n$, we have
  $\psi(d\q)=0$ or $\psi(d\q)>d/(2\varphi(d))$ and that the latter
  occurs for infinitely many $(d,\q)\in\NN\times\PP^n$. Then
  $\Leb(W'(\psi))=1$.
\end{lemma}

\begin{proof}
  If $\psi(d\q)>d/(2\varphi(d))$, then
  \begin{align*}
    \Leb\parens*{A'(d\q, \psi(d\q))}
    &\geq \Leb\parens*{A'(d\q, d/(2\varphi(d)))} \\
    &\geq 2^{-m},
  \end{align*}
  by Lemma~\ref{lem:size}. This implies that $W'(\psi)$ has measure at
  least $2^{-m}$, and so by Theorem~\ref{thm:01} it must have measure $1$. 
\end{proof}

\begin{lemma}\label{lem:redux}
  Suppose $\psi:\ZZ^n\to \RR_{\geq 0}$ and for each $d\geq 1$ and
  $\q\in\PP^n$ we have $\psi(d\q) \leq d/(2\varphi(d))$. If there is some $c>0$ such that
  $\frac{\varphi(d)\Psi(d)}{d} \geq c$ for infinitely many
  $d\geq 1$, then $\Leb\parens*{W'(\psi)}=1$.
\end{lemma}

\begin{proof}
  If for some $d\geq 1$ we have $\Psi(d) = \infty$, then the sets
  $\set*{A_{n,m}'(d\q, \psi(d\q))}_{\q\in\PP^n}$ have a diverging
  measure sum, and by Lemma~\ref{lem:indA}, they are pairwise
  independent. The result follows from the second Borel--Cantelli
  lemma. So let us assume that $\Psi(d)$ is always finite.

  Define for each $d\in\NN$, the set
  \begin{equation*}
    C(d) = \bigcup_{\q\in\PP^n}A_{n,m}'(d\q). 
  \end{equation*}
  By Proposition~\ref{lem:erdoschung},
  \begin{align*}
    \Leb(C(d))
    &\geq \frac{\parens*{\sum_{\q\in\PP^n} \Leb\parens{A_{n,m}'(d\q)}}^2}{\sum_{\q,\br\in\PP^n} \Leb\parens*{A_{n,m}'(d\q)\cap A_{n,m}'(d\br)}} \\
    &\overset{\textrm{Lem.~\ref{lem:indA}}}{\geq} \frac{\parens*{\sum_{\q\in\PP^n} \Leb\parens{A_{n,m}'(d\q)}}^2}{\sum_{\q,\br\in\PP^n} \Leb\parens*{A_{n,m}'(d\q)}\Leb\parens*{A_{n,m}'(d\br)}}\\
    &= \frac{\parens*{\sum_{\q\in\PP^n} \Leb\parens{A_{n,m}'(d\q)}}^2}{\parens*{\sum_{\q\in\PP^n} \Leb\parens{A_{n,m}'(d\q)}}^2 + \sum_{\q\in\PP^n} \Leb\parens{A_{n,m}'(d\q)}}\\
    &\overset{\textrm{Lem.~\ref{lem:size}}}{\geq} \frac{\sum_{\q\in\PP^n} \parens*{\frac{\varphi(d)\psi(d\q)}{d}}^m}{\sum_{\q\in\PP^n} \parens*{\frac{2\varphi(d)\psi(d\q)}{d}}^m + 1}
  \end{align*}
  If $\frac{\varphi(d)\Psi(d)}{d} \geq c$, then this implies
  \begin{equation*}
    \Leb\parens*{C(d)}\geq \frac{c^m}{2^mc^m + 1} > 0.
  \end{equation*}
  Since this happens for infinitely many $d$, it immediately implies
  $\limsup C(d)$ has positive measure, hence $W'(\psi)$ has positive
  measure. The lemma follows by applying Theorem~\ref{thm:01}.
\end{proof}

\section{Proof of Theorem~\ref{thm:multiv}}
\label{sec:proof}

We can dispatch the convergence part quickly. The sets
$A_{n,m}'(\q,\psi(\q))\, (\q\in\ZZ^n)$ have measures bounded above by
$\parens*{\frac{2\varphi(\gcd(\q))\psi(\q)}{\gcd(\q)}}^m$. If the
series in~(\ref{eq:maindiv}) converges, then the Borel--Cantelli lemma
implies $W'(\psi)$ has measure zero.

\

For the divergence part, we are given a function
$\psi: \ZZ_{\geq 0}^n\to\RR_{\geq 0}$ satisfying the divergence
condition~(\ref{eq:div}) and we must
show~(\ref{eq:measureone}). For $d\geq 1$, $\q\in\PP^n$, let
\begin{equation*}
  \psi_1(d\q) =
  \begin{cases}
    \psi(d\q) &\textrm{if } \psi(d\q) \leq \frac{d}{2\varphi(d)}\\
    0 &\textrm{if } \psi(d\q) >\frac{d}{2\varphi(d)}\\
  \end{cases}
\end{equation*}
and $\psi_2= \psi -\psi_1$. If $\psi_2$ satisfies~(\ref{eq:div}), then
$\Leb(W'(\psi_2))=1$, by Lemma~\ref{lem:doverphid}, and we are done
because $W'(\psi_2)\subset W'(\psi)$. Therefore, it suffices to prove
the theorem for $\psi_1$, that is, we may assume that $\psi_2\equiv 0$
and $\psi = \psi_1$. Having made that assumption,
Lemma~\ref{lem:redux} lets us further assume that
\begin{equation}\label{eq:smallPsi}
\frac{\varphi(d)\Psi(d)}{d} < 1/4
\end{equation}
for every $d\geq 1$. Among other things, this implies that for every
$X\geq 1$ there exists some $Y>X$ such that
\begin{equation*}
\parens*{\frac{1}{2}}^{m-1} < \sum_{X\leq d\leq Y}\parens*{\frac{\varphi(d)\Psi(d)}{d}}^m < \parens*{\frac{1}{2}}^{m-2}.
\end{equation*}
Let $X$ be one such large parameter.

Rather than work directly with the sets $A_{n,m}'(d\q, \psi(d\q))$, it
will be convenient to define the subsets
\begin{equation*}
  A_{n,m}''(d\q, \psi(d\q)) = \set*{\x\in\TT^{nm} : \exists \p\in \ZZ^m,\, \p\, (\bmod\, 1)\in P_d, \abs{(d\q)\x-\p} < \psi(d\q)},
\end{equation*}
where $P_d$ is as in Lemma~\ref{lem:size}. It will suffice to show
that the limsup set associated to the sets
$A_{n,m}''(d\q, \psi(d\q))$ has full measure. Notice that since
\begin{equation}\label{eq:doubleprimevprime}
  A_{n,m}''(d\q, \psi(d\q)) \gg A_{n,m}'(d\q, \psi(d\q)),
\end{equation}
the new sets also have a diverging measure sum.
  
Our goal is to show that for $X$ sufficiently large,
\begin{equation}\label{eq:task}
  \Leb\parens*{\bigcup_{X\leq d\leq Y}\bigcup_{\q\in\PP^n}A_{n,m}''(d\q, \psi(d\q))} \gg 1
\end{equation}
with an implicit constant that depends only on $m$. To that end,
observe that by Proposition~\ref{lem:erdoschung},
\begin{multline}\label{eq:echung}
  \Leb\parens*{\bigcup_{X\leq d\leq Y}\bigcup_{\q\in\PP^n}A_{n,m}''(d\q, \psi(d\q))} \\
  \geq \frac{\parens*{\sum_{X\leq d\leq Y}\sum_{\q\in\PP^n}\Leb\parens*{A_{n,m}''(d\q, \psi(d\q))}}^2}{\sum_{X\leq k,\ell \leq Y}\sum_{\q,\br\in\PP^n}\Leb\parens*{A_{n,m}''(k\q, \psi(k\q))\cap A_{n,m}''(\ell\br, \psi(\ell\br))}}.
\end{multline}
By Lemma~\ref{lem:size} we have
\begin{align*}
  \parens*{\sum_{X\leq d\leq Y}\sum_{\q\in\PP^n}\Leb\parens*{A_{n,m}''(d\q, \psi(d\q))}}^2
  &\gg \parens*{\sum_{X\leq d\leq Y}\sum_{\q\in\PP^n}\parens*{\frac{\varphi(d)\psi(d\q)}{d}}^m}^2 \\
  &= \parens*{\sum_{X\leq d\leq Y}\parens*{\frac{\varphi(d)\Psi(d)}{d}}^m}^2 \\
  &\geq \parens*{\frac{1}{2}}^{2(m-1)} \gg 1. 
\end{align*}
For the denominator expression in~(\ref{eq:echung}) we have
\begin{multline*}
  \sum_{X\leq k,\ell \leq Y}\sum_{\q,\br\in\PP^n}\Leb\parens*{A_{n,m}''(k\q, \psi(k\q))\cap A_{n,m}''(\ell\br, \psi(\ell\br))} \\
  =\sum_{X\leq d \leq Y}\sum_{\q \in\PP^n}\Leb\parens*{A_{n,m}''(d\q, \psi(d\q))} \\
  + \underset{k\q\neq\ell\br}{\sum_{X\leq k,\ell \leq Y}\sum_{\q,\br\in\PP^n}}\Leb\parens*{A_{n,m}''(k\q, \psi(k\q))\cap A_{n,m}''(\ell\br, \psi(\ell\br))} \\
  \ll \sum_{X\leq d \leq Y}\parens*{\frac{\varphi(d)\Psi(d)}{d}}^m
  + \underset{k\q\neq\ell\br}{\sum_{X\leq k,\ell \leq Y}\sum_{\q,\br\in\PP^n}}\Leb\parens*{A_{n,m}''(k\q, \psi(k\q))\cap A_{n,m}''(\ell\br, \psi(\ell\br))} \\
  \leq \parens*{\frac{1}{2}}^{m-2} + \underset{k\q\neq\ell\br}{\sum_{X\leq k,\ell \leq
      Y}\sum_{\q,\br\in\PP^n}}\Leb\parens*{A_{n,m}''(k\q,
    \psi(k\q))\cap A_{n,m}''(\ell\br, \psi(\ell\br))}.
\end{multline*}
Therefore, showing~(\ref{eq:task}) reduces to showing
\begin{equation}\label{eq:14}
  \underset{k\q\neq\ell\br}{\sum_{X\leq k,\ell \leq
      Y}\sum_{\q,\br\in\PP^n}}\Leb\parens*{A_{n,m}''(k\q,
    \psi(k\q))\cap A_{n,m}''(\ell\br, \psi(\ell\br))} \ll 1. 
\end{equation}  
We rewrite
\begin{multline*}
  \underset{k\q\neq\ell\br}{\sum_{X\leq k,\ell \leq Y}\sum_{\q,\br\in\PP^n}}\Leb\parens*{A_{n,m}''(k\q, \psi(k\q))\cap A_{n,m}''(\ell\br, \psi(\ell\br))} \\
  = \underset{\q\neq\br}{\sum_{X\leq k,\ell \leq Y}\sum_{\q,\br\in\PP^n}}\Leb\parens*{A_{n,m}''(k\q, \psi(k\q))\cap A_{n,m}''(\ell\br, \psi(\ell\br))} \\
  + \indent 2 \sum_{X\leq k< \ell \leq
    Y}\sum_{\q\in\PP^n}\Leb\parens*{A_{n,m}''(k\q, \psi(k\q))\cap
    A_{n,m}''(\ell\q, \psi(\ell\q))}.
\end{multline*}
By Lemma~\ref{lem:indA}, the first term is bounded above by
\begin{align*}
  \underset{\q\neq\br}{\sum_{X\leq k,\ell \leq Y}\sum_{\q,\br\in\PP^n}}&\Leb\parens*{A_{n,m}'(k\q, \psi(k\q))\cap A_{n,m}'(\ell\br, \psi(\ell\br))} \\
                                                                       &\leq \underset{\q\neq\br}{\sum_{X\leq k,\ell \leq Y}\sum_{\q,\br\in\PP^n}}\Leb\parens*{A_{n,m}'(k\q, \psi(k\q))}\Leb\parens*{A_{n,m}'(\ell\br, \psi(\ell\br))} \\
                                                                       &\ll \parens*{\sum_{X\leq d\leq Y}\parens*{\frac{\varphi(d)\Psi(d)}{d}}^m}^2 \leq \parens*{\frac{1}{2}}^{2(m-2)} \ll 1. 
\end{align*}
Therefore, the task of establishing~(\ref{eq:14}) has reduced to showing
\begin{equation}\label{eq:15}
  \sum_{X\leq k< \ell \leq Y}\sum_{\q\in\PP^n}\Leb\parens*{A_{n,m}''(k\q, \psi(k\q))\cap A_{n,m}''(\ell\q, \psi(\ell\q))} \ll 1
\end{equation}  
with implicit constant depending at most on $m$.

A simple modification of Lemma~\ref{lem:projection} shows that
\begin{equation*}
  \Leb\parens*{A_{n,m}''(k\q, \psi(k\q))\cap A_{n,m}''(\ell\q, \psi(\ell\q))} = \Leb\parens*{A_{1,m}''(k, \psi(k\q))\cap A_{1,m}''(\ell, \psi(\ell\q))}.
\end{equation*}
Notice that, because of~(\ref{eq:smallPsi}), the set
$A_{1,m}''(d, \Psi(d))$ is a union of disjoint balls for every $d\geq 1$. Therefore, by Lemma~\ref{lem:dilation},
\begin{multline*}
  \Leb\parens*{A_{1,m}''(k, \psi(k\q))\cap A_{1,m}''(\ell, \psi(\ell\q))} \\
  \leq \max\set*{\parens*{\frac{\psi(k\q)}{\Psi(k)}}^m,
    \parens*{\frac{\psi(\ell\q)}{\Psi(\ell)}}^m}\Leb\parens*{A_{1,m}''(k,
    \Psi(k))\cap A_{1,m}''(\ell, \Psi(\ell))}.
\end{multline*}
(In fact, this is the sole reason we work with $A''$
instead of $A'$.) Hence,
\begin{multline}\label{eq:11}
  \sum_{X\leq k< \ell \leq Y}\sum_{\q\in\PP^n}\Leb\parens*{A_{n,m}''(k\q, \psi(k\q))\cap A_{n,m}''(\ell\q, \psi(\ell\q))} \\
  \leq \sum_{X\leq k< \ell \leq Y}\Leb\parens*{A_{1,m}''(k, \Psi(k))\cap A_{1,m}''(\ell, \Psi(\ell))} \sum_{\q\in\PP^n} \parens*{\frac{\psi(k\q)^m}{\Psi(k)^m} +  \frac{\psi(\ell\q)^m}{\Psi(\ell)^m}} \\
  \leq 2\sum_{X\leq k< \ell \leq Y}\Leb\parens*{A_{1,m}''(k,
    \Psi(k))\cap A_{1,m}''(\ell, \Psi(\ell))}.
\end{multline}
Let $\I$ be the set of $d\geq 1$ for which $\Psi(d)>1/2$. If
$k,\ell\in\I$, then
\begin{equation*}
  \max\set*{2k\Psi(\ell), 2\ell\Psi(k)}\geq \max\set{k,\ell}\geq \sqrt{k\ell}.
\end{equation*}
By~(\ref{eq:smallPsi}), we also have $2\Psi(d) \leq d/(2\varphi(d))$,
so Proposition~\ref{lem:PVoverlaps} applies. Combining it
with~(\ref{eq:powerm}) and~(\ref{eq:powermint}), it tells us that
\begin{align*}
  \Leb\parens*{A_{1,m}'(k, 2\Psi(k))\cap A_{1,m}'(\ell, 2\Psi(\ell))}
  &\ll 4^m\parens*{\frac{\varphi(k)\Psi(k)}{k}}^m\parens*{\frac{\varphi(k)\Psi(k)}{k}}^m \\
  &\ll \Leb\parens*{A_{1,m}'(k, \Psi(k))}\Leb\parens*{A_{1,m}'(\ell, \Psi(\ell))},
\end{align*}
where again we have used~(\ref{eq:smallPsi}) in this last comparison. Hence,
\begin{align*}
  \Leb\parens*{A_{1,m}''(k, \Psi(k))\cap A_{1,m}''(\ell, \Psi(\ell))}
  & \leq \Leb\parens*{A_{1,m}'(k, 2\Psi(k))\cap A_{1,m}'(\ell, 2\Psi(\ell))}\\
  & \ll \Leb\parens*{A_{1,m}'(k, \Psi(k))}\Leb\parens*{A_{1,m}'(\ell, \Psi(\ell))} \\
    & \ll \Leb\parens*{A_{1,m}''(k, \Psi(k))}\Leb\parens*{A_{1,m}''(\ell, \Psi(\ell))},
\end{align*}
by~(\ref{eq:doubleprimevprime}). If in addition we had
\begin{equation*}
  \sum_{d\in\I}\parens*{\frac{\varphi(d)\Psi(d)}{d}}^m = \infty,
\end{equation*}
then we could have restricted $\psi(d\q)$ to $(d,\q)\in\I\times\PP^n$
at the outset. With~(\ref{eq:11}), we would have
\begin{multline*}
  \sum_{X\leq k< \ell \leq Y}\sum_{\q\in\PP^n}\Leb\parens*{A_{n,m}''(k\q, \psi(k\q))\cap A_{n,m}''(\ell\q, \psi(\ell\q))} \\
  \leq 2\sum_{X\leq k< \ell \leq Y}\Leb\parens*{A_{1,m}''(k,
    \Psi(k))}\Leb\parens*{A_{1,m}''(\ell, \Psi(\ell))} \\
  \ll \parens*{\sum_{X\leq d \leq
      Y}\parens*{\frac{\varphi(d)\Psi(d)}{d}}^m}^2 \leq \parens*{\frac{1}{2}}^{2(m-2)} \ll 1,
\end{multline*}
establishing~(\ref{eq:15}),~(\ref{eq:14}), and~(\ref{eq:task}), which
is our goal.

We may therefore assume that $\Psi(d)\leq 1/2$ for all $d\geq 1$. In
this case, we apply Proposition~\ref{prop:KMVP} to get
\begin{equation*}
  \sum_{X\leq k\leq \ell\leq Y}\Leb\parens*{A_{1,m}''(k, \Psi(k))\cap A_{1,m}''(\ell, \Psi(\ell))} \ll 1.
\end{equation*}
With~(\ref{eq:11}), this implies~(\ref{eq:15}), which
implies~(\ref{eq:14}) and in turn~(\ref{eq:task}). As a consequence,
we find that $W'(\psi)$ has measure bounded below by the implicit
constant in~(\ref{eq:task}). In particular, it has positive measure,
so it must have full measure, by Theorem~\ref{thm:01}. This completes
the proof of Theorem~\ref{thm:multiv}.\qed

\section{Proof of Theorem~\ref{thm:catlin}}\label{sec:proof-theor-refthm:c}

The cases of Theorem~\ref{thm:catlin} where $m>1$ were proved by
Beresnevich and Velani in~\cite{BVKG} and the case $m=n=1$ follows
from Koukoulopoulos--Maynard, so it is only left to establish the
cases where $m=1$ and $n>1$. 

\

In this case, notice that we have
\begin{equation}\label{eq:4}
\Phi_1(\q)\asymp\frac{\varphi(\gcd(\q))}{\gcd(\q)}\abs{\q}.
\end{equation}
Let $\overline \psi(\q) = \sup_{t\geq 1}\frac{\psi(t\q)}{t}$.  We can assume without loss of generality that
\begin{equation}\label{eq:13}
\psi(\q)/\abs{\q}\to 0 \quad\textrm{as}\quad \abs{\q}\to\infty.
\end{equation}
Otherwise, one can easily show that
$W(\psi) = W(\overline\psi)=[0,1]^{n}$.

We first consider the convergence part of
Theorem~\ref{thm:catlin}. Suppose
\begin{equation}\label{eq:2}
\abs{\q\cdot\x - p} < \psi(\q)
\end{equation}
and let $t = \gcd(p,\q)$. Then
\begin{equation*}
\abs{\q'\cdot\x - p'} < \frac{\psi(t\q')}{t},
\end{equation*}
where $\q':=\q/t$ and $p':=p/t$, and therefore we have
\begin{equation*}
\abs{\q'\cdot \x-p'} < \overline\psi(\q')
\end{equation*}
where $\gcd(p', \q')=1$. This, together with the
assumption~(\ref{eq:13}), implies that
$W(\psi) \subset W'(\overline\psi)$. Note that the convergence of the
series in~(\ref{eq:catlindiv}) is equivalent to the convergence of the
series in~(\ref{eq:maindiv}) applied to $\overline\psi$. Therefore, by
Theorem~\ref{thm:multiv}, $\Leb(W'(\overline\psi))=0$ and so
$\Leb(W(\psi))=0$.

Now we will show that the divergence part of the $m=1$ cases of
Theorem~\ref{thm:catlin} are equivalent to the $m=1$ cases of a
version of Theorem~\ref{thm:multiv} where the
requirement~(\ref{eq:coprimality}) is removed. We follow the argument
in~\cite[Appendix]{randomfractions} where the equivalence is shown in
the case $(m,n)=(1,1)$.

We claim that
\begin{equation}\label{eq:catlinpsi}
  W(\psi) = W(\overline\psi).
\end{equation}
Since $\psi \leq \overline\psi$, we automatically have
$W(\psi)\subset W(\overline\psi)$. Now suppose
$\x\in W(\overline \psi)$. Then there are infinitely many $(p,\q)$
for which
\begin{equation}\label{eq:3}
  \abs{\q\cdot\x- p} < \overline\psi(\q)
\end{equation}
holds. For a given $\q\in\ZZ^n\setminus{0}$,
let $t_\q$ be an integer for which
$\overline \psi(\q) = \frac{\psi(t_\q\q)}{t_\q}$. (Such an integer is
guaranteed to exist because we have assumed that $\psi(\q)/\abs{\q}$
approaces $0$.) Note that if $(p,\q)$ solves~(\ref{eq:3}), then
\begin{equation*}
  \abs{t_\q\q\cdot\x-t_\q p} < \psi(t_\q\q). 
\end{equation*}
In this way, infinitely many solutions to~(\ref{eq:3}) give rise to
infinitely many solutions to~(\ref{eq:2}). This establishes
$W(\overline \psi)\subset W(\psi)$, which proves~(\ref{eq:catlinpsi}).

Now, assume the weak form of Theorem~\ref{thm:multiv} where the
coprimality condition~(\ref{eq:coprimality}) is removed. Suppose
$\psi:\ZZ^n\to\RR_{\geq 0}$ satisfies the divergence
condition~(\ref{eq:catlindiv}). Then, recalling~(\ref{eq:4}), we have
that $\overline\psi$ satisfies the divergence
condition~(\ref{eq:maindiv}). Therefore, by our assumption, we have
$\Leb\parens*{W(\overline\psi)}=1$, hence by~(\ref{eq:catlinpsi}),
$\Leb\parens*{W(\psi)}=1$, and the divergence part of
Theorem~\ref{thm:catlin} holds.

Conversely, suppose Theorem~\ref{thm:catlin} holds, and suppose
$\psi$ satisfies the divergence condition~(\ref{eq:maindiv}). Since
$\psi \leq \overline\psi$, we must also have that $\overline\psi$
satisfies~(\ref{eq:maindiv}), that is, $\psi$
satisfies~(\ref{eq:catlindiv}). Therefore, $\Leb\parens*{W(\psi)}=1$,
by Theorem~\ref{thm:catlin}. This proves the weak form of
Theorem~\ref{thm:multiv} in the cases where $m=1$. \qed

\bibliographystyle{plain}


\begin{thebibliography}{10}

\bibitem{AistleitnerDS}
Christoph Aistleitner.
\newblock A note on the {D}uffin--{S}chaeffer conjecture with slow divergence.
\newblock {\em Bulletin of the London Mathematical Society}, 46(1):164--168,
  2014.

\bibitem{AistleitneretalExtraDivergence}
Christoph Aistleitner, Thomas Lachmann, Marc Munsch, Niclas Technau, and
  Agamemnon Zafeiropoulos.
\newblock The {D}uffin-{S}chaeffer conjecture with extra divergence.
\newblock {\em Adv. Math.}, 356:106808, 11, 2019.

\bibitem{AllenBeresnevich}
Demi Allen and Victor Beresnevich.
\newblock A mass transference principle for systems of linear forms and its
  applications.
\newblock {\em Compos. Math.}, 154(5):1014--1047, 2018.

\bibitem{BBDV}
Victor Beresnevich, Vasily Bernik, Maurice Dodson, and Sanju Velani.
\newblock Classical metric {D}iophantine approximation revisited.
\newblock In {\em Analytic number theory}, pages 38--61. Cambridge Univ. Press,
  Cambridge, 2009.

\bibitem{BHHVextraii}
Victor Beresnevich, Glyn Harman, Alan Haynes, and Sanju Velani.
\newblock The {D}uffin-{S}chaeffer conjecture with extra divergence {II}.
\newblock {\em Math. Z.}, 275(1-2):127--133, 2013.

\bibitem{BVmassTP}
Victor Beresnevich and Sanju Velani.
\newblock A mass transference principle and the {D}uffin-{S}chaeffer conjecture
  for {H}ausdorff measures.
\newblock {\em Ann. of Math. (2)}, 164(3):971--992, 2006.

\bibitem{BVzeroone}
Victor Beresnevich and Sanju Velani.
\newblock A note on zero-one laws in metrical {D}iophantine approximation.
\newblock {\em Acta Arith.}, 133(4):363--374, 2008.

\bibitem{BVKG}
Victor Beresnevich and Sanju Velani.
\newblock Classical metric {D}iophantine approximation revisited: the
  {K}hintchine-{G}roshev theorem.
\newblock {\em Int. Math. Res. Not. IMRN}, (1):69--86, 2010.

\bibitem{CatlinConj}
Paul~A. Catlin.
\newblock Two problems in metric {D}iophantine approximation. {I}.
\newblock {\em J. Number Theory}, 8(3):282--288, 1976.

\bibitem{DickinsonVelaniCrelle}
Detta Dickinson and Sanju~L. Velani.
\newblock Hausdorff measure and linear forms.
\newblock {\em J. Reine Angew. Math.}, 490:1--36, 1997.

\bibitem{duffinschaeffer}
R.~J. Duffin and A.~C. Schaeffer.
\newblock Khintchine's problem in metric {D}iophantine approximation.
\newblock {\em Duke Math. J.}, 8:243--255, 1941.

\bibitem{ErdosDS}
P.~Erd{\H{o}}s.
\newblock On the distribution of the convergents of almost all real numbers.
\newblock {\em J. Number Theory}, 2:425--441, 1970.

\bibitem{Gallagherkt}
P.~X. Gallagher.
\newblock Metric simultaneous diophantine approximation. {II}.
\newblock {\em Mathematika}, 12:123--127, 1965.

\bibitem{Gallagher01}
Patrick Gallagher.
\newblock Approximation by reduced fractions.
\newblock {\em J. Math. Soc. Japan}, 13:342--345, 1961.

\bibitem{Groshev}
A.~Groshev.
\newblock A theorem on a system of linear forms.
\newblock {\em Doklady Akademii Nauk SSSR}, 19:151--152, 1938.

\bibitem{Harman}
Glyn Harman.
\newblock {\em Metric number theory}, volume~18 of {\em London Mathematical
  Society Monographs. New Series}.
\newblock The Clarendon Press Oxford University Press, New York, 1998.

\bibitem{HPVextra}
Alan~K. Haynes, Andrew~D. Pollington, and Sanju~L. Velani.
\newblock The {D}uffin-{S}chaeffer conjecture with extra divergence.
\newblock {\em Math. Ann.}, 353(2):259--273, 2012.

\bibitem{Khintchineonedimensional}
A.~Khintchine.
\newblock Einige {S}\"atze \"uber {K}ettenbr\"uche, mit {A}nwendungen auf die
  {T}heorie der {D}iophantischen {A}pproximationen.
\newblock {\em Math. Ann.}, 92(1-2):115--125, 1924.

\bibitem{Khintchine}
A.~Khintchine.
\newblock Zur metrischen {T}heorie der diophantischen {A}pproximationen.
\newblock {\em Math. Z.}, 24(1):706--714, 1926.

\bibitem{KMDS}
Dimitris Koukoulopoulos and James Maynard.
\newblock On the {D}uffin-{S}chaeffer conjecture.
\newblock {\em Ann. of Math. (2)}, 192(1):251--307, 2020.

\bibitem{PVbordeaux}
A.~D. Pollington and R.~C. Vaughan.
\newblock The {$k$}-dimensional {D}uffin and {S}chaefer conjecture.
\newblock {\em S\'{e}m. Th\'{e}or. Nombres Bordeaux (2)}, 1(1):81--88, 1989.

\bibitem{PollingtonVaughan}
A.~D. Pollington and R.~C. Vaughan.
\newblock The {$k$}-dimensional {D}uffin and {S}chaeffer conjecture.
\newblock {\em Mathematika}, 37(2):190--200, 1990.

\bibitem{randomfractions}
Felipe~A. Ram\'{\i}rez.
\newblock Khintchine's theorem with random fractions.
\newblock {\em Mathematika}, 66(1):178--199, 2020.

\bibitem{Sprindzuk}
Vladimir~G. Sprind\v{z}uk.
\newblock {\em Metric theory of {D}iophantine approximations}.
\newblock Scripta Series in Mathematics. V. H. Winston \& Sons, Washington,
  D.C.; John Wiley \& Sons, New York-Toronto, Ont.-London, 1979.
\newblock Translated from the Russian and edited by Richard A. Silverman, With
  a foreword by Donald J. Newman.

\bibitem{Strauch}
Oto Strauch.
\newblock Some new criterions for sequences which satisfy {D}uffin-{S}chaeffer
  conjecture. {I}.
\newblock {\em Acta Math. Univ. Comenian.}, 42/43:87--95 (1984), 1983.

\bibitem{Vaaler}
Jeffrey~D. Vaaler.
\newblock On the metric theory of {D}iophantine approximation.
\newblock {\em Pacific J. Math.}, 76(2):527--539, 1978.

\bibitem{Vilchinski}
V.~T. Vil'\v{c}inski\u{\i}.
\newblock On simultaneous approximations by irreducible fractions.
\newblock {\em Vests\={\i} Akad. Navuk BSSR Ser. F\={\i}z.-Mat. Navuk},
  (2):41--47, 140, 1981.

\end{thebibliography}

\end{document}